\documentclass[a4paper,11pt]{amsart}

\usepackage{enumerate}

\usepackage{upref}

\usepackage[colorlinks,backref]{hyperref}                   

\usepackage{amssymb,amsmath}

\newtheorem{theorem}{Theorem}[section]
\newtheorem{lemma}[theorem]{Lemma}

\newtheorem{definition}{Definition}[section]

\theoremstyle{remark} 
\theoremstyle{definition} 

\numberwithin{equation}{section}

\newif\ifcomment \commentfalse \def\commentON{\commenttrue}
 \long\outer\def\BC#1\EC{\ifcomment \sloppy \par
\# \ldots\dotfill {\em #1} \dotfill \# \par \fi } \commentON

\newcommand{\remove}[1]{}

\newcommand{\eps}{\ensuremath{\varepsilon}}

\newcommand{\R}{{{\mathbb{R}}}}

{%

\begin{enumerate}}%
{\end{enumerate}
}

\begin{document}

\title[A two-species cross attraction system]{ On a two-species  cross attraction system in higher dimensions}

\author[{E. A. Carlen}]{Eric A. Carlen}
\address[ Eric A. Carlen]{\newline
	Department of Mathematics\newline
        Rutgers University \newline
	Piscataway, NJ, USA}
\email{carlen@math.rutgers.edu}
\urladdr{http://www.math.rutgers.edu/~carlen/}

\author[{S. Ulusoy}]{Suleyman Ulusoy}
\address[ Suleyman Ulusoy]{\newline
	Department of Mathematics and Natural Sciences\newline
        American University of Ras Al Khaimah\newline
	PO Box 10021, Ras Al Khaimah, UAE}
\email{suleyman.ulusoy@aurak.ac.ae}
\urladdr{https://aurak.ac.ae/en/dr-suleyman-ulusoy/\newline
$https://www.researchgate.net/profile/Suleyman_{-}Ulusoy$}




\date{\today}

\begin{abstract}
We consider a degenerate chemotaxis model with two-species and
two-stimuli in dimension $d \geq 3$. Under the hypothesis of integrable initial data with finite second moment and energy, we show local-in-time existence for any mass of free-energy solutions, namely weak solutions with some free energy estimates. We exhibit that the qualitative behavior of solutions is decided by a set of  critical values: there is a critical value of a parameter pair  in the system of  equations for which there is a global-in-time energy solution and   there exist blowing-up free-energy solutions under a criticality condition is violated for the  parameter pair.

\end{abstract}

\maketitle


\section{Introduction}\label{sec:intro}

Inspired by \cite{C-L}(see also \cite{20, 30}), for modeling the interaction and motion of two cell populations in breast cancer cell invasion models in $\R^d, \, d \geq 3$ we propose the following chemotaxis kind of system with two chemicals and a nonlinear diffusion term:
\begin{equation}\label{a1}
\begin{split}
\partial_t u &= \alpha_1 \left( \frac{d-2}{d} \right) ||u||_{\frac{2d}{d+2}}^{4/(d+2)} \Delta(u^{\frac{2d}{d+2}}) -  \nabla \cdot \left( u \nabla v\right), \quad  x \in \R^d, t > 0, \\
-\Delta v &= w,  \quad  x \in \R^d, t > 0,\\
\partial_t w &= \alpha_2 \left( \frac{d-2}{d} \right) ||w||_{\frac{2d}{d+2}}^{4/(d+2)} \Delta(w^{\frac{2d}{d+2}}) -  \nabla \cdot \left( w \nabla z\right),  \quad  x \in \R^d, t > 0,  \\
-\Delta z &= u,   \quad  x \in \R^d, t > 0,
\end{split}
\end{equation}
with initial data satisfying
\begin{equation}\label{a2}
u(x, 0) = u_{0}(x); \quad w(x, 0) = w_{0}(x), \, x \in \R^d.
\end{equation}
In \eqref{a1} $u(x,t)$ and $w(x,t)$ denote the macrophages and the tumor cells; $v(x,t)$ and $z(x,t)$ denote the concentration of the chemicals produced by $w(x,t)$ and $u(x,t)$, respectively and    $\alpha_1, \alpha_2 \in (0,1]$ . For simplicity   the initial data are assumed to satisfy
\begin{equation}\label{a3}
u_0 \in L^1(\R^d; (1+|x|^2)dx) \cap L^{\infty}(\R^d), \, \nabla u_0^{\frac{2d}{d+2}} \in L^2(\R^d), \, \, u_0 \geq 0,
\end{equation}
\begin{equation}\label{a4}
w_0 \in L^1(\R^d; (1+|x|^2)dx) \cap L^{\infty}(\R^d), \, \nabla w_0^{\frac{2d}{d+2}} \in L^2(\R^d), \, \, w_0 \geq 0.
\end{equation}

We note that $v(x,t)$ and $z(x,t)$ are given by
\begin{equation}\label{a5}
v(x, t) = \mathcal{K} \ast w = c_d \iint_{\R^d \times \R^d} \frac{w(y,t)}{|x-y|^{d-2}} \, dy,
\end{equation}
\begin{equation}\label{a6}
z(x, t) = \mathcal{K} \ast u = c_d \iint_{\R^d \times \R^d}  \frac{u(y,t)}{|x-y|^{d-2}} \, dy,
\end{equation}
with $\mathcal{K}(x) = \frac{c_d}{|x|^{d-2}}$ and $c_d$ is the surface area of the sphere $S^{d-1}$ in $\R^d$. We can rewrite the system \eqref{a1} as follows:
\begin{equation}\label{a7}
\begin{split}
\partial_t u &= \alpha_1 \left( \frac{d-2}{d}\right) ||u||_{\frac{2d}{d+2}} \Delta(u^{\frac{2d}{d+2}}) -  \nabla \cdot \left( u \nabla \mathcal{K} \ast w \right),\\
\partial_t w &= \alpha_2 \left( \frac{d-2}{d}\right) ||w||_{\frac{2d}{d+2}} \Delta(w^{\frac{2d}{d+2}}) -  \nabla \cdot \left( w \nabla \mathcal{K} \ast u \right),
\end{split}
\end{equation}
with initial data \eqref{a2}. Note that mass is conserved for \eqref{a1}(or \eqref{a7}):
\begin{equation}\label{a8}
M_1 := \int_{\R^d} u(x, t) \, dx = \int_{\R^d} u_0(x) \, dx; \quad M_2 := \int_{\R^d}  w(x,t) \, dx = \int_{\R^d} w_0(x) \, dx.
\end{equation}
There is a free energy associated with equation \eqref{a1}:
\begin{equation}\label{a9}
E_\alpha[u(t), w(t)] := \alpha_1 ||u||_{\frac{2d}{d+2}}^2 + \alpha_2 ||w||_{\frac{2d}{d+2}}^2 - c_d    H[u, w],
\end{equation}
where
\begin{equation}\label{a10}
H[u, w] :=  \iint_{\R^d \times \R^d} \frac{u(x, t) w(y,t)}{|x-y|^{d-2}}   \, dx \, dy.
\end{equation}
The energy functional can be rewritten as
\begin{equation}\label{a9b}
E_\alpha[u(t), w(t)] := (\sqrt{\alpha_1} ||u||_{\frac{2d}{d+2}} - \sqrt{\alpha_2} ||w||_{\frac{2d}{d+2}})^ 2   + \left( 2\sqrt{\alpha_1\alpha_2}||u||_{\frac{2d}{d+2}}||w||_{\frac{2d}{d+2}}  - c_d    H[u, w]\right)\ .
\end{equation}
By the HLS inequality, (see \eqref{b3} below), the second term is positive for all $u,w$ in case
\begin{equation}\label{a9c}
\frac{ 2\sqrt{\alpha_1\alpha_2}}{c_d} \geq C_{HLS}\ ,
\end{equation}
the precise value of which was determined by Lieb for the exponents considered here.   Suppose however, that \eqref{a9c} is violated. Let
$$\delta := C_{HLS} - \frac{ 2\sqrt{\alpha_1\alpha_2}}{c_d}\ .$$
Then if we choose $u$ and $w$ to be HLS optimizers such that
$$
\sqrt{\alpha_1} ||u||_{\frac{2d}{d+2}} = \sqrt{\alpha_2} ||w||_{\frac{2d}{d+2}}\ ,
$$
$$
E_\alpha[u(t), w(t)]  = - \delta ||u||_{\frac{2d}{d+2}}||w||_{\frac{2d}{d+2}}\ .
$$
Thus, \eqref{a9c} is the necessary and sufficient condition for non-negativity of $E_\alpha[u(t), w(t)]$.

The connection between the free energy \eqref{a9} and the equation \eqref{a1} is that the latter  can be written as a gradient flow with respect to Wasserstein metric,  and  then formally one has the dissipation of this free energy
\begin{equation}\label{a11}
\begin{split}
\frac{d}{dt}E_\alpha[u(t), w(t)] &= - \int_{\R^d} u\left|  \nabla \left( \alpha_1  2||u||_{\frac{2d}{d+2}}^{\frac{4}{d+2}} u^{\frac{d-2}{d+2}}  \right) -  \nabla v  \right|^2   \, dx \\
&\quad - \int_{\R^d} w \left|  \nabla \left( \alpha_2 2||w||_{\frac{2d}{d+2}}^{\frac{4}{d+2}} w^{\frac{d-2}{d+2}} \right) -  \nabla z  \right|^2   \, dx.
\end{split}
\end{equation}

We now define the weak and free energy solutions for \eqref{a1}-\eqref{a2}.

\begin{definition}\label{def:weaksol}

Let $d \geq 3$  and $T > 0$. Suppose $(u_0, w_0)$ satisfies  \eqref{a3}-\eqref{a4}. Then $(u, w)$	 of  nonnegative functions defined in $\R^d \times (0, T)$ is called a weak solution if

\begin{enumerate}[(i)]

\item  $ (u, w)  \in \left(  C([0,T); L^1(\R^d))  \cap L^{\infty}(\R^d \times (0, T))  \right)^2 $  and \\ $\left(u^{\frac{2d}{d+2}}, w^{\frac{2d}{d+2}}\right)  \in \left(  L^2(0, T; H^1(\R^d))\right)^2$.

 \item $(u, w)$ satisfies

\begin{equation}\label{a12}
\begin{split}
&\int_0^T \int_{\R^d} u \phi_{1t} \,  dx \,   dt  +  \int_{\R^d}  u_0(x) \phi_{1}(x, 0) \, dx   \\
& = \int_0^T \int_{\R^d} \left( \alpha_1 2||u||_{\frac{2d}{d+2}}^{\frac{4}{d+2}}  \nabla(u^{\frac{2d}{d+2}}) -  u \nabla v \right)\cdot \nabla \phi_1 \, dx \, dt.\\
& \int_0^T \int_{\R^d} w \phi_{2t} \,  dx \,   dt  +  \int_{\R^d}  w_0(x) \phi_{2}(x, 0) \, dx   \\
& = \int_0^T \int_{\R^d} \left(\alpha_2 2||w||_{\frac{2d}{d+2}}^{\frac{4}{d+2}}  \nabla(w^{\frac{2d}{d+2}}) -  w \nabla z \right)\cdot \nabla \phi_2 \, dx \, dt,
\end{split}
\end{equation}
for any test functions $\phi_i \in \mathcal{D}(\R^d \times [0, T));  i = 1, 2$ with $v = \mathcal{K} \ast w$ and $z = \mathcal{K} \ast u.$

\end{enumerate}

\end{definition}

\begin{definition}\label{def:free-energy-sol}
Let $T > 0$. Then $(u, w)$ is called a free energy solution with initial data $(u_0, w_0)$ on $(0, T)$ if $(u, w)$ is a weak solution and moreover satisfies
\begin{equation}\label{a13}
\left( u^{\frac{3d-2}{2(d+2)}}, w^{\frac{3d-2}{2(d+2)}} \right) \in \left( L^2(0, T; H^1(\R^d))\right)^2,
\end{equation}
and
\begin{equation}\label{a14}
\begin{split}
&E_\alpha[u(t), w(t)] + \int_0^t \int_{\R^d} u \left| \nabla \left( \alpha_1 2||u||_{\frac{2d}{d+2}}^{\frac{4}{d+2}} u^{\frac{d-2}{d+2}} \right) -  \nabla v\right|^2 \, dx \, ds \\
& + \int_0^t \int_{\R^d} w \left| \nabla \left(\alpha_2 2||w||_{\frac{2d}{d+2}}^{\frac{4}{d+2}} w^{\frac{d-2}{d+2}} \right) -  \nabla z\right|^2 \, dx \, ds \leq E_\alpha[u_0, w_0],
\end{split}
\end{equation}
for all $t \in (0, T)$ with $v = \mathcal{K} \ast w$ and $z = \mathcal{K} \ast u$.
\end{definition}

\begin{theorem}\label{thm:existence}
Suppose that the initial data $(u_0, w_0)$ with $||u_0||_1 = M_1,  ||w_0||_1 = M_2$ satisfies \eqref{a3} and \eqref{a4}.

\begin{enumerate}[(i)]

\item If $\alpha_1, \alpha_2$ satisfy \eqref{a9c}, then  there exists a global free energy solutions.

\item If $\alpha_1, \alpha_2$ do not satisfy \eqref{a9c}, then it is possible to construct large initial data ensuring blow up in finite time.

\end{enumerate}

\end{theorem}

\section{ Approximate System}\label{sec:regularization}

For the existence of solutions as usual we first consider a regularized system

\begin{equation}\label{b1}
\begin{split}
\partial_t u_{\eps} &= \alpha_1  \left( \frac{d-2}{d} \right) ||u_{\eps}||_{\frac{2d}{d+2}}^{4/(d+2)} \Delta((u_{\eps} + \eps)^{\frac{2d}{d+2}}) -  \nabla \cdot \left( u_{\eps} \nabla v_{\eps}\right), \, \,  x \in \R^d, t > 0, \\
 v_{\eps} &= \mathcal{K} \ast w_{\eps},  \quad  x \in \R^d, t > 0,\\
\partial_t w_{\eps} &= \alpha_2 \left(  \frac{d-2}{d} \right) ||w_{\eps}||_{\frac{2d}{d+2}}^{4/(d+2)} \Delta((w_{\eps}+\eps)^{\frac{2d}{d+2}}) -  \nabla \cdot \left( w_{\eps} \nabla z_{\eps}\right),  \,  \,   x \in \R^d, t > 0,  \\
z_{\eps} &= \mathcal{K} \ast u_{\eps},   \quad  x \in \R^d, t > 0,
\end{split}
\end{equation}
with initial data satisfying
\begin{equation}\label{b2}
u_{\eps}(x, 0) = u_{0}^{\eps}(x) \geq 0; \quad w_{\eps}(x, 0) = w_{0}^{\eps}(x) \geq 0, \, x \in \R^d.
\end{equation}
with $u_0^{\eps}$ and $w_0^{\eps}$ being the convolution of $u_0$ and $w_0$ with a sequence of mollifiers and $||u_0^\eps||_1 = ||u||_1 = M_1$ and $||w_0^\eps||_1 = ||w||_1 = M_2$. As usual we will obtain a priori estimates for the regularized system \eqref{b1}-\eqref{b2}  to ensure the existence of weak or free energy solution as $\eps$ tends to 0.

By following the general procedure in the single population chemotaxis systems  \cite{43, ulusoy, carlen-ulusoy} one can  obtain the following lemma for which no proof is provided here.

\begin{lemma}\label{lemma2.1}
Let $d \geq 3$. There exists  $T_{max}^{\eps}  \in (0, \infty]$ denoting the maximal existence time such that the regularized system \eqref{b1}-\eqref{b2} has a unique nonnegative solution $(u_\eps, w_\eps) \in \left(W_p^{2,1}(Q_T)\right)^2$ with some $p>1$, where $Q_T = \R^d \times (0, T)$ with $T \in (0, T_{max}^{\eps}$ and
\begin{equation*}
W_p^{2,1}(Q_T) := \{u \in L^p(0, T; W^{2,p}(\R^d) \cap W^{1, p}(0, T; L^p(\R^d)))\}.
\end{equation*}
Moreover, if $T_{max}^{\eps} < \infty$ then
\begin{equation*}
\lim_{t \to T_{max}^{\eps} } \left( ||u_\eps (\cdot, t)||_{\infty} +   ||w_\eps (\cdot, t)||_{\infty}  \right)  = \infty.
\end{equation*}

\end{lemma}

We also recall here a version of the Hardy-Littlewood-Sobolev inequality(HLS) that if
\begin{equation*}
 \frac{1}{p} + \frac{1}{q}  = 1+ \frac{\lambda}{d}
\end{equation*}
and $h_1 \in L^p(\R^d)$,  $h_2 \in L^q(\R^d)$ with $p, q > 1$ then there exists a constant $C_{HLS} = C_{HLS}(d, \lambda, p) > 0$ such that
\begin{equation}\label{b3}
\left|  \iint_{\R^d \times \R^d}  \frac{h_1(x)h_2(y)}{|x-y|^{d-\lambda}}  \, dy \, dx  \right|  \leq C_{HLS} ||h_1||_p||h_2||_q.
\end{equation}

\begin{lemma}\label{lemma2.3}
Let $T \in (0, T_{max}^\eps]$. Suppose that there exists a constant  $C>0$ such that  solution $(u_\eps, w_\eps)$ of the system \eqref{b1}-\eqref{b2} with initial data $(u_0^\eps, w_0^\eps)$ being the convolution of $(u_0, w_0)$ satisfies $||u_\eps(t)||_{\frac{2d}{d+2}} \leq C$   and $||w_\eps(t)||_{\frac{2d}{d+2}} \leq C$ for $t \in (0, T)$. Then there exists a constant $C = C(d, u_0^\eps, w_0^\eps) > 0$ such that
\begin{equation}\label{b4}
||\left(u_\eps(t), w_\eps(t)\right)||_r \leq C,
\end{equation}
and
\begin{equation}\label{b5}
||\left(v_\eps(t), z_\eps(t)\right)||_r + ||\left(\nabla v_\eps(t), \nabla z_\eps(t)\right)||_r \leq C,
\end{equation}
for $r \in [1, \infty)$  and   $t \in (0, T)$.

\end{lemma}

\begin{proof}

For $p>1$, we test  the first equation in \eqref{b1} by $u_\eps^{p-1}$ and integrate to find that
\begin{equation}\label{b6}
\begin{split}
\frac{1}{p}\frac{d}{dt} \int_{\R^d} u_\eps^p \, dx  &=   \underbrace{- \alpha_1 (\frac{d-2}{d})||u_\eps||_{\frac{2d}{d+2}}^{\frac{4}{d+2}} \int_{\R^d} \nabla(u_\eps^{p-1}) \cdot \nabla(u_\eps + \eps)^{\frac{2d}{d+2}} \, dx}_{I} \\
& \quad + \underbrace{ \int_{\R^d}   u_\eps \nabla (u_\eps^{p-1}) \cdot \nabla v_\eps \, dx}_{II}.
\end{split}
\end{equation}
It is not difficult to see that
\begin{equation}\label{b7}
\begin{split}
I &= -2 \alpha_1\frac{d-2}{d+2}(p-1)||u_\eps||_{\frac{2d}{d+2}}^{\frac{4}{d+2}} \int_{\R^d} u_\eps^{p-2}(u_\eps + \eps)^{\frac{2d}{d+2}-1}|\nabla u_\eps|^2 \, dx \\
&\leq -2 \alpha_1 \frac{d-2}{d+2}(p-1)||u_\eps||_{\frac{2d}{d+2}}^{\frac{4}{d+2}} \int_{\R^d} u_\eps^{p+ \frac{2d}{d+2}-3}|\nabla u_\eps|^2 \, dx \\
&=-4 \alpha_1 \frac{d-2}{d+2}(p-1) \frac{1}{p+\frac{d-2}{d+2}}\int_{\R^d} \left|  \nabla\left(u_\eps^{\frac{p+(\frac{d-2}{d+2})}{2}} \right) \right|^2 \, dx.
\end{split}
\end{equation}
\begin{equation}\label{b8}
II =  \frac{p-1}{p}  \int_{\R^d} u_\eps^p w_\eps \, dx.
\end{equation}
Similarly, by multiplying the third equation in \eqref{b1}  by $w_\eps^{p-1}$ and using $-\Delta z_\eps = u_\eps$
\begin{equation}\label{b9}
\begin{split}
\frac{1}{p} \frac{d}{dt} \int_{\R^d} w_\eps^p \, dx &\leq  -4 \alpha_2 \frac{d-2}{d+2}(p-1) \frac{1}{p+\frac{d-2}{d+2}}\int_{\R^d} \left|  \nabla\left(w_\eps^{\frac{p+(\frac{d-2}{d+2})}{2}} \right) \right|^2 \, dx\\
& \qquad +  \frac{p-1}{p}  \int_{\R^d} u_\eps  w_\eps^p \, dx.
\end{split}
\end{equation}
Combining together
\begin{equation}\label{b10}
\begin{split}
&\frac{1}{p} \left(  \frac{d}{dt}\left\{  \int_{\R^d} u_\eps^p \, dx +   \int_{\R^d} w_\eps^p \, dx  \right\} \right) \\
&+ 4\alpha_1 \frac{d-2}{d+2}(p-1) \frac{1}{p+\frac{d-2}{d+2}}\int_{\R^d} \left|  \nabla\left(u_\eps^{\frac{p+(\frac{d-2}{d+2})}{2}} \right) \right|^2 \, dx \\
&+ 4\alpha_2 \frac{d-2}{d+2}(p-1) \frac{1}{p+\frac{d-2}{d+2}}\int_{\R^d} \left|  \nabla\left(w_\eps^{\frac{p+(\frac{d-2}{d+2})}{2}} \right) \right|^2 \, dx \\
&\leq   \frac{p-1}{p}  \left \{\int_{\R^d} u_\eps^p w_\eps \, dx +  \int_{\R^d} u_\eps  w_\eps^p \, dx \right\}.
\end{split}
\end{equation}
Now,
\begin{equation}\label{b11}
\int_{\R^d} u_\eps^p w_\eps \, dx \leq \left(  \int_{\R^d} u_\eps^{pr_1} \, dx  \right)^{\frac{1}{r_1}} \left( \int_{\R^d} w_\eps^{r_1'} \, dx \right)^{\frac{1}{r_1'}}
\end{equation}
and
\begin{equation}\label{b12}
\int_{\R^d} u_\eps w_\eps^p \, dx \leq \left(  \int_{\R^d} u_\eps^{r_2} \, dx  \right)^{\frac{1}{r_2}} \left( \int_{\R^d} w_\eps^{pr_2'} \, dx \right)^{\frac{1}{r_2'}}
\end{equation}
by H\"older inequality with $r_1, r_2 > 1, r_1' = \frac{r_1}{r_1-1}$ and $r_2' = \frac{r_2}{r_2-1}$. We have
\begin{equation}\label{b13}
\frac{2d}{d+2} < pr_1 < \frac{\left( p+\frac{d-2}{d+2}\right)d}{d-2}
\end{equation}
and
\begin{equation}\label{b14}
\frac{1}{r_1} > \max \left\{  \frac{d-2}{2d}, \frac{d-2}{d}\frac{p}{p+\frac{d-2}{d+2}}  \right\}.
\end{equation}
We now recall a variant of Gagliardo-Nirenberg inequality (see Lemma 6 in \cite{45})
\begin{equation}\label{b15}
||\Phi||_{k_2} \leq C^{2/{r+m-1}} ||\Phi||_{k_1}^{1-\sigma} ||\nabla \Phi^{\frac{r+m-1}{2}}||_2^{\frac{2\sigma}{r+m-1}}
\end{equation}
with $m \geq 1,  \,  k_1 \in [1, r+m-1]$ and $1 \leq k_1 \leq k_2  \leq \frac{(r+m-1)d}{d-2}$ with $d \geq 3$ and $\sigma = \frac{r+m-1}{2}\left( \frac{1}{k_1} - \frac{1}{k_2}\right)\left(  \frac{1}{d} -\frac{1}{2} + \frac{r+m-1}{2k_1}  \right)^{-1}$. We pick $r=p,  m  = \frac{2d}{d+2},  k_1 = \frac{2d}{d+2},  k_2 = p r_1$ and use the above inequalities to deduce
\begin{equation}\label{b16}
||u_\eps||_{pr_1}^p   \leq C ||u_\eps||_{\frac{2d}{d+2}}^{p(1-\sigma)} ||\nabla u_\eps^{\frac{p+\frac{d-2}{d+2}}{2}}||_2^{p\frac{2\sigma}{p+\frac{d-2}{d+2}}},
\end{equation}
with
\begin{equation*}
\sigma = \frac{p+\frac{d-2}{d+2}}{2} \frac{\frac{d+2}{2d}- \frac{1}{pr_1}}{\frac{1}{d}-\frac{1}{2} + \frac{p+\frac{d-2}{d+2}}{\frac{4d}{d+2}}} \in (0, 1).
\end{equation*}
Using the bound  $||u_\eps(t)||_{\frac{2d}{d+2}} \leq C$ for  $t \in (0, T)$  we deduce that
\begin{equation}\label{b17}
\left( \int_{\R^d}  u_\eps^{pr_1} \, dx  \right)^{1/{r_1}}    \leq C || \nabla u_\eps^{\frac{p+\frac{d-2}{d+2}}{2}}||_2^{\frac{\frac{p(d+2)}{2d} - \frac{1}{r_1}  }{\frac{1}{d} - \frac{1}{2} + \frac{p+\frac{d-2}{d+2}}{\frac{4d}{d+2}}}  }.
\end{equation}
Likewise, using
\begin{equation*}
\frac{2d}{d+2} < r_1' < \frac{(p+\frac{d-2}{d+2})d}{d-2}
\end{equation*}
and Gagliardo-Nirenberg inequality with $||w_\eps (t)||_{\frac{2d}{d+2}} \leq C$
\begin{equation}
||w_\eps ||_{r_1'} \leq C || \nabla w_\eps^{\frac{p+\frac{d-2}{d+2}}{2}}||_2^{\frac{\frac{d+2}{2d}-\frac{1}{r_1'}}{\frac{1}{d}-\frac{1}{2}+ \frac{p+\frac{d-2}{d+2}}{4d/{d+2}}}}.
\end{equation}
Then

\begin{equation}\label{b18}
\begin{split}
&\left( \int_{\R^d} u_\eps^{pr_1} \, dx\right)^{1/{r_1}}   \left( \int_{\R^d}  w_\eps^{r_1'} \, dx \right)^{1/{r_1'}} \\
& \leq C ||\nabla u_\eps^{\frac{p+\frac{d-2}{d+2}}{2}}||_2^{\frac{p(\frac{d+2}{2d})-\frac{1}{r_1}}{\frac{1}{d}-\frac{1}{2} + \frac{p+\frac{d-2}{d+2}}{4d/{d+2}}}}   ||\nabla w_\eps^{\frac{p+\frac{d-2}{d+2}}{2}}||_2^{\frac{p(\frac{d+2}{2d})-1+\frac{1}{r_1}}{\frac{1}{d}-\frac{1}{2} + \frac{p+\frac{d-2}{d+2}}{4d/{d+2}}}}.
\end{split}
\end{equation}

We have
$$ \frac{2d}{d+2}  < r_2  < \frac{(p+\frac{d-2}{d+2})d}{d-2}.$$
Then by Gagliardo-Nirenberg inequality
\begin{equation}\label{b19}
\left( \int_{\R^d} u_\eps^{r_2} \, dx \right)^{1/{r_2}} \leq C ||\nabla u_\eps^{\frac{p+\frac{d-2}{d+2}}{2}}||_2^{\frac{\frac{d+2}{2d}-\frac{1}{r_2}}{\frac{1}{d}-\frac{1}{2}+ \frac{p+\frac{d-2}{d+2}}{4d/{d+2}}}}
\end{equation}
as $||u_\eps(t)||_{\frac{2d}{d+2}} \leq C$ and $||w_\eps(t)||_{\frac{2d}{d+2}} \leq C$.
Now,
$$ \frac{2d}{d+2} < p r_2' < \frac{(p+\frac{d-2}{d+2})d}{d-2}.$$
One has
\begin{equation}\label{b20}
\begin{split}
||w_\eps||_{pr_2'}^p &\leq C ||w_\eps||_{\frac{2d}{d+2}}^{p(1-\sigma)}||\nabla w_\eps^{\frac{p+\frac{d-2}{d+2}}{2}}||_2^{p\frac{2\sigma}{p+\frac{d-2}{d+2}}}\\
& \leq C \left |\left|\nabla w_\eps^{\frac{p+\frac{d-2}{d+2}}{2}}\right|\right|_2^{\frac{\left(\frac{p(d+2)}{2d}-\frac{1}{r_2'}\right)}{\left(\frac{1}{d}-\frac{1}{2}+\frac{p+\frac{d-2}{d+2}}{4d/{d+2}}\right)}}
\end{split}
\end{equation}
with
$$\sigma = \frac{(d+2)\frac{p+\frac{d-2}{d+2}}{4d} -\frac{1}{pr_2'}}{\frac{1}{d} - \frac{1}{2}+(d+2)\frac{ p+\frac{d-2}{d+2}}{4d}}.$$
Then,
\begin{equation}\label{b21}
\begin{split}
\left( \int_{\R^d} u_\eps^{r_2} \, dx \right)^{1/{r_2}} \left( \int_{\R^d} w_\eps^{pr_2'} \, dx \right)^{1/{r_2'}} &\leq  C \left|\left| \nabla u_\eps^{p+\frac{d-2}{d+2}} \right|\right|_2^{\frac{\frac{d+2}{2d}-\frac{1}{r_2}}{\frac{1}{d}-\frac{1}{2} + (d+2)\frac{p+\frac{d-2}{d+2}}{4d} }}\\
& \quad \times \left|\left| \nabla w_\eps^{p+\frac{d-2}{d+2}} \right|\right|_2^{\frac{\frac{d+2}{2d}-\frac{1}{r_2}-1}{\frac{1}{d}-\frac{1}{2} + (d+2)\frac{p+\frac{d-2}{d+2}}{4d} }}
\end{split}
\end{equation}
Combining, one gets
\begin{equation}\label{b22}
\begin{split}
\frac{1}{p}\frac{d}{dt}\left( \int_{\R^d} u_\eps^p + w_\eps^p \, dx \right)  + \delta_1(d, \alpha_1)\int_{\R^d} |\nabla u_\eps^{\frac{p+\frac{d-2}{d+2}}{2}}|^2 \, dx \\
 + \delta_2(d, \alpha_2)\int_{\R^d} |\nabla w_\eps^{\frac{p+\frac{d-2}{d+2}}{2}}|^2 \, dx \\
 \leq (\frac{p-1}{p}) \left( \int_{\R^d} u_\eps^{pr_1} \, dx \right)^{1/{r_1}} \left( \int_{\R^d} w_\eps^{r_1'} \, dx \right)^{1/{r_1'}}\\
 + (\frac{p-1}{p}) \left( \int_{\R^d} u_\eps^{r_2} \, dx \right)^{1/{r_2}} \left( \int_{\R^d} w_\eps^{qr_2'} \, dx \right)^{1/{r_2'}}\\
\leq C ||\nabla u_\eps^{\frac{p+\frac{d-2}{d+2}}{2}}||_2^{\frac{\frac{p(d+2)}{2d}-\frac{1}{r_1}}{\frac{1}{d}-\frac{1}{2}}}
||\nabla w_\eps^{\frac{p+\frac{d-2}{d+2}}{2}}||_2^{\frac{\frac{2-d}{2d}+ \frac{1}{r_1}}{\frac{1}{d} - \frac{1}{2}  + \frac{(d+2)(p+\frac{d-2}{d+2})}{4d}}}
\\
+  C || \nabla u_\eps^{\frac{p+\frac{d-2}{d+2}}{2}}||_2^{\frac{\frac{d+2}{2d} - \frac{1}{r_2}}{\frac{1}{d}-\frac{1}{2}+ \frac{(d+2)(p+\frac{d-2}{d+2})}{4d} }}
  || \nabla w_\eps^{\frac{p+\frac{d-2}{d+2}}{2}}||_2^{\frac{\frac{p(d+2)}{2d}-1+\frac{1}{r_2}}{\frac{1}{d} - \frac{1}{2} + \frac{(d+2)(p+\frac{d-2}{d+2})}{4d} }},
\end{split}
\end{equation}
where $\delta_1, \delta_2 > 0$ and can be calculated explicitly.\\

We now deal with the boundedness of $u_\eps$ and $w_\eps$  in $L^p-$space: Let $\gamma_1 > 0 , \gamma_2 > 0$ be such that $\gamma_1 + \gamma_2 < 2$. For $\eps > 0$ by Young's inequality we have
$$ \alpha^{\gamma_1}\beta^{\gamma_2} \le \eps(\alpha^2+\beta^2) + C.$$
From the above calculations, there exists some $p > \bar{p}$ with some $\bar{p}>1$ such that
$$ \frac{\frac{p(d+2)}{2d} - \frac{1}{r_1}}{\frac{1}{d} - \frac{1}{2} + \frac{p+\frac{2d}{d+2} -1}{4d/(d+2)}}  +  \frac{\frac{d+2}{2d} -1 + \frac{1}{r_1}}{\frac{1}{d} - \frac{1}{2} + \frac{p+ \frac{2d}{d+2} -1}{4d/(d+2)}}  < 2.$$
$$ \frac{\frac{d+2}{2d} - \frac{1}{r_2}}{\frac{1}{d} - \frac{1}{2} + \frac{p+\frac{2d}{d+2}-1}{4d/(d+2)}}  +  \frac{\frac{(d+2)p}{2d} -1 + \frac{1}{r_2} }{ \frac{1}{d}  - \frac{1}{2}  + \frac{ p+ \frac{2d}{d+2} -1  }{4d/(d+2)} }  < 2.   $$
The above two inequalities can be simplified further(left to the reader). Now,
\begin{equation}\label{b23}
 \begin{split}
&\frac{1}{p} \left( \int_{\R^d} u_\eps^p \, dx +  \int_{\R^d} w_\eps^p \, dx \right) + \delta_1(d, \alpha_1) \int_{\R^d} \left|  \nabla u_{\eps}^{\frac{p+\frac{d-2}{d+2}}{2}}  \right|^2 \, dx \\
&\qquad + \delta_2(d, \alpha_2)\int_{\R^d} \left|  \nabla w_{\eps}^{\frac{p+\frac{d-2}{d+2}}{2}}  \right|^2 \, dx \leq C,
\end{split}
\end{equation}
with $\delta_1, \delta_2 > 0$. Using Gagliardo-Nirenberg inequality with $||u||_1 = M_1$ and $||w||_1 = M_2$ and Young's inequality
\begin{equation}\label{b24}
\begin{split}
\frac{1}{p} ||u_\eps||_p^p &\leq  C ||\nabla u_\eps^{\frac{p+\frac{d-2}{d+2}}{2}}||_2^{\frac{p-1}{\frac{1}{d} - \frac{1}{2} + \frac{p+ \frac{d-2}{d+2}}{2} }}\\
& \quad \leq \frac{\frac{4d}{d+2}(p-1)}{(p+\frac{d-2}{d+2})^2} \int_{\R^d} |\nabla u_\eps^{\frac{p+\frac{d-2}{d+2}}{2}}|^2 \, dx + C
\end{split}
\end{equation}
and
\begin{equation}\label{b25}
\frac{1}{p} \int_{\R^d} w_\eps^p \, dx \leq   \frac{\frac{4d}{d+2}(p-1)}{(p+\frac{d-2}{d+2})^2} \int_{\R^d} |\nabla w_\eps^{\frac{p+\frac{d-2}{d+2}}{2}}|^2 \, dx + C,
\end{equation}
by
$$ \frac{p-1}{\frac{1}{d} - \frac{1}{2} + \frac{p+\frac{d-2}{d+2}}{2}} < 2.$$
Writing $y(t) := \frac{1}{p} \left( \int_{\R^d} (u_\eps^p + w_\eps^p) \ dx \right) $ we get
$$ y'(t)+ y(t)  \leq C \quad \text{for} \quad t \in (0, T).$$
Then,
$$ ||u_\eps(t)||_p, \, \,  ||w_\eps(t)||_p \leq C \quad \text{for} \quad  r \in [1,\infty)\quad \text{and}\quad t \in (0, T).$$
From this we deduce that
\begin{equation}\label{b26}
\left|\left|(u_\eps(t), w_\eps(t))\right|\right|_r \leq C, \quad \text{for} \quad r \in [1, \infty) \quad \text{and}\quad t \in (0, T).
\end{equation}
We now try to improve the regularities of $v$ and $z$. Since
$$v_\eps(x, t) = \mathcal{K} \ast w_\eps = c_d \iint_{\R^d \times \R^d} \frac{w_\eps(y,t)}{|x-y|^{d-2}} \, dy, \quad  z_\eps(x, t)  = c_d \iint_{\R^d \times \R^d}  \frac{u_\eps(y,t)}{|x-y|^{d-2}} \, dy,$$
by HLS inequality we have
\begin{equation}\label{b27}
\begin{split}
||\nabla v_\eps||_r &\leq c_d(d-2)||I_1(w_\eps)||_r \leq C ||w_\eps||_{\frac{dr}{d+r}}\\
||\nabla z_\eps||_r &\leq  C ||u_\eps||_{\frac{dr}{d+r}}.
\end{split}
\end{equation}
On the other hand, Calderon-Zygmud inequality implies that
\begin{equation}\label{b28}
\begin{split}
||\partial_{x_i}\partial_{x_j} v_\eps||_r &\leq C ||w_\eps||_r \\
||\partial_{x_i}\partial_{x_j} z_\eps||_r &\leq C ||u_\eps||_r, \,\, 1 \leq i,j \leq d.
\end{split}
\end{equation}
The above estimates and the Morrey's inequality imply that
\begin{equation}\label{b29}
||(v_\eps(t), z_\eps(t))||_r + ||(\nabla v_\eps(t), \nabla z_\eps(t))||_r \leq C, \quad \text{for} \quad r \in [1, \infty] \,\,\, \text{and} \,\,\, t \in (0, T).
\end{equation}

\end{proof}

\begin{lemma}\label{lem:2.4}
Under the assumptions of Lemma \ref{lemma2.3}, there exists $C >0$ independent of $\eps$ such that the strong solution of \eqref{b1} satisfies
\begin{equation}\label{b30}
||(u_\eps(t), w_\eps(t))||_{\infty} \leq C, \quad \forall t \in (0, T).
\end{equation}
Moreover, there exists a global weak solution $(u, w)$ of \eqref{a1}-\eqref{a2} which  also satisfies a uniform bound.
\end{lemma}

\begin{proof}
Using Lemma \ref{lemma2.3} one can apply the Moser's iteration to obtain a priori  estimate of solution in $L^\infty.$ Then this solution can be extended globally in time from the extensibility criterion in Lemma \ref{lemma2.1} establishing \eqref{b30}, one can refer to Proposition 10 of \cite{45}. From \eqref{b30} there is $(u,v,w,z)$ with regularities given in Definition \ref{def:weaksol} such that, up to a subsequence, $\eps \to 0,$
\begin{equation}\label{eqn:converge}
\begin{split}
u_{\eps_n}  &\to u \, \, \text{strongly} \, \, \text{in}  \, \, C([0, T); L_{loc}^p(\R^d)) \, \, \text{and a.e. in} \, \, \R^d \times (0, T),\\
\nabla u_{\eps_n}^{2d/(d+2)} &\rightharpoonup \nabla u^{2d/(d+2)} \, \, \text{weakly}-* \, \, \text{in} \, \, L^{\infty}((0,T); L^2(\R^d)),\\
w_{\eps_n} &\to w \, \, \text{strongly} \, \, \text{in}  \, \, C([0, T); L_{loc}^p(\R^d)) \, \, \text{and a.e. in} \, \, \R^d \times (0, T),\\
\nabla w_{\eps_n}^{2d/(d+2)} &\rightharpoonup \nabla w^{2d/(d+2)} \, \, \text{weakly}-* \, \, \text{in} \, \, L^{\infty}((0,T); L^2(\R^d)),\\
v_{\eps_n}(t) &\to   v(t)  \, \, \text{strongly  in} \, \, L_{loc}^r \, \, \text{and a.e. in} \, \, (0, T),\\
\nabla v_{\eps_n}(t) &\to \nabla  v(t) \, \, \text{strongly in} \, \, L_{loc}^r(\R^d) \, \, \text{and a.e. in} (0, T),\\
\Delta v_{\eps_n}(t) &\rightharpoonup \Delta v(t) \, \, \text{weakly in} \, \, L_{loc}^r(\R^d) \, \, \text{and a.e. in} (0, T),\\
z_{\eps_n}(t) &\to   z(t)  \, \, \text{strongly  in} \, \, L_{loc}^r \, \, \text{and a.e. in} \, \, (0, T),\\
\nabla z_{\eps_n}(t) &\to \nabla z(t) \, \, \text{strongly in} \, \, L_{loc}^r(\R^d) \, \, \text{and a.e. in} (0, T),\\
\Delta z_{\eps_n}(t) &\rightharpoonup \Delta z(t) \, \, \text{weakly in} \, \, L_{loc}^r(\R^d) \, \, \text{and a.e. in} (0, T),\\
\end{split}
\end{equation}
where $p \in (1, \infty), $ $r \in (1, \infty]$ and $T \in (0, \infty).$ Since the above convergence can be established following \cite[Section 4]{44}, and the details are left to the reader. Therefore, we have a global weak solution $(u, v, w, z)$ over $\R^d \times (0, T)$ with $T > 0.$

\end{proof}

We now follow \cite{43} to establish that a global weak solution is also a global free energy solution.

\begin{lemma}\label{lem:2.5}
Consider a global weak solution in Lemma \ref{lem:2.4}. Then it is also a global free energy solution $(u, w)$ of \eqref{a1} given in Definition \ref{def:free-energy-sol}.

\end{lemma}

\begin{proof}
The proof can  be done following the proof of Proposition 2.1 in \cite{BCL} and \cite{43}. The details are omitted.
\end{proof}

\begin{theorem}\label{thm:4.1}
Under the assumption  \eqref{a3}-\eqref{a4}  on the  initial data with $||u_0||_1 = M_1,  \, \,   ||w_0||_1=M_2$, there exists a $T_{max} \in (0, \infty]$ and a free energy solution $(u, w)$ over $\R^d \times (0, T_{max})$ of \eqref{a1} such that either $T_{max} = \infty$  or $T_{max} < \infty$ and
\begin{equation}\label{b-31}
\lim_{t \to T_{max}}  \left( ||u(\cdot, t)||_{\infty}  +  ||w(\cdot, t)||_{\infty} \right) = \infty.
\end{equation}
Moreover, let $\alpha_1, \alpha_2$ satisfy \eqref{a9c}, Then, if $T_{max} < \infty$
\begin{equation}\label{b-32}
\lim_{t \to T_{max}} ||u(\cdot, t)||_{\frac{2d}{d+2}} = \infty = \lim_{t \to T_{max}} ||u(\cdot, t)||_{\frac{2d}{d+2}}.
\end{equation}

\end{theorem}

\begin{proof}
For $(u_0, w_0)$ satisfying \eqref{a3}-\eqref{a4} local existence and \eqref{b-31} can be established as in the proof of Theorem 1.1. of \cite{43} by employing approximation arguments.
\begin{equation}\label{b-33}
\alpha_1||u||_{\frac{2d}{d+2}}^2 + \alpha_2 ||w||_{\frac{2d}{d+2}}^2  \leq c_d  H[u, w] + E_{\alpha}[u_0, w_0].
\end{equation}
Modifying Lemma 3.1  of  \cite{C-L} we deduce that

\begin{equation}\label{b-34}
\begin{split}
\alpha_1||u||_{\frac{2d}{d+2}}^2 &+ \alpha_2 ||w||_{\frac{2d}{d+2}}^2  \leq c_d\eta ||u||_{\frac{2d}{d+2}}^{\frac{2d}{d+2}} + c_d C \eta^{-\frac{d+2}{d-2}}M_2^{\frac{4d^2}{(d+2)^2}+ \frac{8d}{(d-2)(d+2)}}||w||_{\frac{2d}{d+2}}^{\frac{4-6d}{(d-2)^2}} \\
& + E_{\alpha}[u_0, w_0]  \leq c_d\eta||u||_{\frac{2d}{d+2}}^{\frac{2d}{d+2}} + c_d C \eta^{-\frac{d+2}{d-2}}||w||_{\frac{2d}{d+2}}^{\frac{2d}{d+2}} + C,
\end{split}
\end{equation}
where we have used \eqref{a9c} in the last inequality to discard $E_{\alpha}[u_0, w_0]$.
Taking $\eta$ small enough, one has

\begin{equation}\label{b-35}
||u(t)||_{\frac{2d}{d+2}}^{\frac{2d}{d+2}} \leq C ||w(t)||_{\frac{2d}{d+2}}^{\frac{2d}{d+2}} + C, \quad \text{for} \, \, t \in (0, T_{max}),
\end{equation}
and if $\eta$ is large enough one observes that
\begin{equation}\label{b-36}
||w(t)||_{\frac{2d}{d+2}}^{\frac{2d}{d+2}} \leq C' ||u(t)||_{\frac{2d}{d+2}}^{\frac{2d}{d+2}}  + C', \quad \text{for} \, \, t \in (0, T_{max}).
\end{equation}
Hence, \eqref{b-32} holds by \eqref{b-31}, \eqref{b-35} and \eqref{b-36}.

\end{proof}

\begin{theorem}\label{them:4.2}
Let $d \geq 3$ and let $(u_0, w_0)$ be initial data with $||u_0||_1 = M_1$ and $||w_0||_1= M_2$ satisfying \eqref{a3}-\eqref{a4}. Then, if $\alpha_1, \alpha_2$ satisfy \eqref{a9c}, then \eqref{a1} has a global free energy solution given in Definition  \ref{def:free-energy-sol}.

\end{theorem}

\begin{proof}
We have

\begin{equation}\label{b-37}
\begin{split}
|H[u,w]| &\leq \frac{(d+2)}{2c_d(d-2)} ||u||_{\frac{2d}{d+2}}^{\frac{2d}{d+2}}  \\
	   &+C||w||_1^{\frac{4d^2+ 8d \frac{d+2}{d-2} }{(d-2)^2}} ||w||_{\frac{2d}{d+2}}^{\frac{2d}{d-2}}	\\
             &\leq \frac{d+2}{2c_d(d-2)} \left[ ||u||_{\frac{2d}{d+2}}^{\frac{2d}{d+2}} +  ||w||_{\frac{2d}{d+2}}^{\frac{2d}{d+2}}  \right] + C,
\end{split}
\end{equation}
using the Young's inequality. Combining this with
$$ \alpha_1  ||u||_{\frac{2d}{d+2}}^{2} +  \alpha_2 ||w||_{\frac{2d}{d+2}}^{2}  \leq c_d H[u, w] + E_{\alpha}[u_0, w_0]$$
one deduces that
$$  ||u||_{\frac{2d}{d+2}} \leq C \quad \text{and} \quad    ||w||_{\frac{2d}{d+2}}^2 \leq C,$$
and this implies, by employing Theorem \ref{thm:4.1}, the existence of a free energy solution.

\end{proof}

\section{BLOW-UP}\label{sec:blowup}

In this section we deal with the finite-time blow-up phenomenon. We begin with a simple lemma stating the time-evolution of second moment of solutions.

\begin{lemma}\label{lem:2ndmonetevolution}
Let $(u_0, w_0)$ satisfy \eqref{a3} and \eqref{a4} and let $(u, w)$ be a free energy solution of \eqref{a1} on $[0, T_{max})$ with $T_{max}  \in (0, \infty]$. Then,
\begin{equation*}
\frac{d}{dt} I(t) = G(t), \quad \forall  t \in (0, T_{max}),
\end{equation*}
where
\begin{equation*}
I(t) := \int_{\R^d} |x|^2 \left( u(x,t) + w(x,t)\right) \, dx
\end{equation*}
and
\begin{equation*}
\begin{split}
G(t)  &:= C(\alpha_1, u) \int_{\R^d} u(x,t)^{\frac{2d}{d+2}} \, dx + C(\alpha_2, w)\int_{\R^d} w(x,t)^{\frac{2d}{d+2}} \, dx \\
&-2c_d(d-2)\iint_{\R^d \times \R^d} \frac{u(x,t)w(y,t)}{|x-y|^{d-2}} \, dy \, dx.
\end{split}
\end{equation*}
with
\begin{equation*}
\begin{split}
& C(\alpha_1, u):= 2(d-2)\alpha_1||u||_{\frac{2d}{d+2}}^{4/{d+2}},\\
&C(\alpha_2, w) := 2(d-2)\alpha_2||w||_{\frac{2d}{d+2}}^{4/{d+2}}.
\end{split}
\end{equation*}
\end{lemma}
Notice that with the definitions,
\begin{equation}\label{bu1}
G(t) = 2(d-2)E_\alpha[u(t),v(t)]\ .
\end{equation}

\begin{proof}
Straightforward calculations lead to
\begin{equation*}
\begin{split}
&\frac{d}{dt} I(t) = \int_{\R^d} |x|^2 \left(  u(x,t) + w(x,t)  \right) \, dx\\
&= \int_{\R^d} |x|^2 \Big[C(\alpha_1, u)\Delta \left(u^{\frac{2d}{d+2}}\right)- \nabla \cdot \left(u \nabla v\right) \\
&\qquad  \qquad \qquad + C(\alpha_2, w)\Delta \left(w^{\frac{2d}{d+2}}\right)- \nabla \cdot \left(w \nabla z\right) \Big]\, dx \\
&= 2dC(\alpha_1, u) \int_{\R^d} u(x,t)^{\frac{2d}{d+2}} \, dx + C(\alpha_2, w)\int_{\R^d} w(x,t)^{\frac{2d}{d+2}} \, dx \\
&\quad + 2 \iint_{\R^d \times \R^d} \left[   x \cdot \nabla \mathcal{K}(x-y)  \right]  u(x,t)w(y,t) \, dy \, dx \\
&\qquad + 2 \iint_{\R^d \times \R^d} \left[   x \cdot \nabla \mathcal{K}(x-y)  \right]  u(y,t)w(x,t) \, dy \, dx,
\end{split}
\end{equation*}
with $\mathcal{K}(x) = \frac{c_d}{|x|^{d-2}}$ and $c_d$ is the surface area of the sphere $S^{d-1}$ in $\R^d$.
Now,
\begin{equation*}
\begin{split}
&2\iint_{\R^d \times \R^d} \left[x \cdot  \nabla \mathcal{K}(x-y) \right] u(x,t) w(y,t) \,dy \, dx \\
& = -2c_d(d-2)\iint_{\R^d \times \R^d} \frac{(x-y) \cdot x}{|x-y|^d} u(x,t) w(y,t) \,dy \, dx \\
& = -2c_d(d-2)\iint_{\R^d \times \R^d} \frac{|x|^2}{|x-y|^d} u(x,t) w(y,t) \,dy \, dx \\
& \quad + 2c_d(d-2)\iint_{\R^d \times \R^d} \frac{y \cdot x}{|x-y|^d} u(x,t) w(y,t) \,dy \, dx \\
&= -c_d(d-2)\iint_{\R^d \times \R^d} \frac{|x|^2}{|x-y|^d} u(x,t) w(y,t) \,dy \, dx \\
& \quad - c_d(d-2)\iint_{\R^d \times \R^d} \frac{|y|^2}{|x-y|^d} u(x,t) w(y,t) \,dy \, dx \\
& \qquad + 2c_d(d-2)\iint_{\R^d \times \R^d} \frac{y \cdot x}{|x-y|^d} u(x,t) w(y,t) \,dy \, dx.
\end{split}
\end{equation*}
Similarly,
\begin{equation*}
\begin{split}
&2\iint_{\R^d \times \R^d} \left[x \cdot  \nabla \mathcal{K}(x-y) \right] u(y,t) w(x,t) \,dy \, dx \\
&= -c_d(d-2)\iint_{\R^d \times \R^d} \frac{|x|^2}{|x-y|^d} u(y,t) w(x,t) \,dy \, dx \\
& \quad - c_d(d-2)\iint_{\R^d \times \R^d} \frac{|y|^2}{|x-y|^d} u(x,t) w(y,t) \,dy \, dx \\
& \qquad + 2c_d(d-2)\iint_{\R^d \times \R^d} \frac{y \cdot x}{|x-y|^d} u(x,t) w(y,t) \,dy \, dx.
\end{split}
\end{equation*}
Combining, we deduce that
\begin{equation*}
\begin{split}
&\frac{d}{dt}\int_{\R^d} |x|^2 \left( u(x,t)+ w(y,t) \right) \, dx = 2dC(\alpha_1, u) \int_{\R^d} u^{\frac{2d}{d+2}} \, dx \\
&\qquad \qquad +  2dC(\alpha_2, w) \int_{\R^d} w^{\frac{2d}{d+2}} \, dx \\
&-c_d(d-2)\iint_{\R^d \times \R^d} \frac{|x|^2+|y|^2}{|x-y|^d} u(x,t)w(y,t) \, dy \, dx \\
&-c_d(d-2)\iint_{\R^d \times \R^d} \frac{|x|^2+|y|^2}{|x-y|^d} u(y,t)w(x,t) \, dy \, dx \\
&+ 4c_d(d-2)\iint_{\R^d \times \R^d} \frac{y \cdot x}{|x-y|^d} u(x,t) w(y,t) \,dy \, dx\\
& =2dC(\alpha_1, u) \int_{\R^d} u^{\frac{2d}{d+2}} \, dx +  2dC(\alpha_2, w) \int_{\R^d} w^{\frac{2d}{d+2}} \, dx \\
&\qquad -2c_d(d-2)\iint_{\R^d \times \R^d} \frac{u(x,t) w(y,t) }{|x-y|^{d-2}} \, dy \, dx.
\end{split}
\end{equation*}

\end{proof}

We now observe that
\begin{equation}\label{q1}
\begin{split}
\frac{d}{dt}I(t) = G(t)  &= 2(d-2)E_{\alpha}[u(t), w(t)]\\
	               &\leq   2(d-2)E_{\alpha}[u_0, w_0] = G(0).		
\end{split}
\end{equation}
As we have seen above, $\alpha_1$ and $\alpha_2$ are such that  \eqref{a9c} is violated,  then we may take the initial data
 $(u_0, w_0)$ to consist of appropriate multiples of $HLS$ optimizers, and then we shall have $E_\alpha[u_0,w_0] < 0$, and hence $G(0) < 0$.   By the montonicity if the energy, we will then have
 $G(t) \leq G(0) < 0$ for all $t$ such that a regular solution exists. But then \eqref{q1} will imply that the second moment will be negative after some time which is impossible  by the non-negativity of $u$ and $w$.  Hence blow-up occurs in a finite time.

\begin{theorem}\label{thm:blow-up}
Let $\alpha_1, \alpha_2$ be such that \eqref{a9c} is violated. Then one can find some initial data $(u_0, w_0)$  satisfying \eqref{a3} and \eqref{a4} such that free energy solution $(u, w)$ of \eqref{a1} with $(u, w)|_{t=0} = (u_0, w_0)$ blows up in finite time.

\end{theorem}

\begin{proof}
We consider an initial data $(u_0, w_0)$ with $G(0) < 0.$ By a continuity argument there exists a $T^* > 0$ such that
$$ G(t) < \frac{G(0)}{2} \quad \text{for all} \quad t \in [0, T^*].$$
By Lemma  \ref{lem:2ndmonetevolution}
\begin{equation}\label{q2}
\frac{d}{dt} I(t)  = G(t) < \frac{G(0)}{2}, \quad \text{for all} \quad t \in [0, T^*].
\end{equation}
Integrating \eqref{q2}, it follows that
\begin{equation}\label{q3}
I(T^*) < I(0) +  \frac{G(0)}{2}T^*.
\end{equation}
We may choose the initial data in such a way that the right hand side of \eqref{q3} is negative. But this leads to a contradiction since it implies that $I(T^*) < 0$ but $I(t)$ is always non-negative for all $t>0.$ Hence the solution $(u(t), w(t))$ blows-up in finite time.

\end{proof}

\section{Acknowledgments} \label{ack}

The work of E. A. Carlen   is partially supported by U.S. N.S.F. grant DMS DMS-2055282.


\begin{thebibliography}{10}
%

\bibitem{B}
W.~Beckner.
\newblock  Sharp Sobolev inequalities on the sphere and the Moser-Trudinger inequality,
\newblock{em Ann. Math},  138 (1)  213--242, 1993
%

\bibitem{BCL}
A.~Blanchet, J.~A.~ Carrillo, and P.~Lauren\c cot.
\newblock Critical mass for a Patlak-Keller-Segel model with degenerate diffusion in higher dimensions.
\newblock {\em Calc. Var. PDE}, 35: 133--168, 2009.

\bibitem{BDEF}
A.~Blanchet, J.~Dolbeault, M.~Escobeda and J~ Fern\'anadez
\newblock Functional inequalities, thick tails and asymptotics for the critical mass Patlak-Keller-Segel model.
\newblock {\em J. Math. Anal. Appl.}, 361: 533-542, 2010.

\bibitem{BCP}
A.~Blanchet, J.~Dolbeault  and J~ Perthame
\newblock Two-dimensional Keller-Segel model: Optimal critical mass and qualitative properties of the solutions.
\newblock {\em Elec. Jour. Diff. Eq.}, {\bf 44}, 2006.

\bibitem{CL}
E.~Carlen, M.~Loss
\newblock Competing symmetries, the logarithmic HLS inequality and Onofri's inequality on $S^n$.
\newblock {\em Geom. Funct. Anal. 2},  9--104, 1992.


%
%







\bibitem{carlen-ulusoy}
E.~A.~ Carlen and S.~ Ulusoy
\newblock Dissipation for a non-convex gradient flow problem of a Patlack-Keller-Segel type for densities on $\R^d, d \geq 3$.
\newblock {\em  Nonlinear Analysis: TMA},  Volume 208, 112314, 2021.





\bibitem{C-L}
J.~A.~ Carrillo,  K.~Lin.
\newblock Sharp conditions on global existence and blow-up in a degenerate two-species and cross-attraction system.
\newblock {\em arXiv:2012.10789 }, 2020.
%
%
%
%


\bibitem{17}
J.~ Dolbeault, and B.~ Perthame.
\newblock Optimal critical mass in two-dimensional Keller-Segel model in $\R^2$.
\newblock {\em C. R. Math. Acad. Sci. Paris}, 339, 611--616, 2004.





\bibitem{20}
E. ~Espejo, K. ~Vilches, C. ~Conca.
\newblock A simultaneous blow-up problem arising in tumor modeling.
\newblock {\em J. Math. Biol.}, 79, 1357--1399, 2019.


\bibitem{JL}
W.~J\"ager, S. Luckhaus,
\newblock On explosions of solutions to a system of partial differential equations modelling chemotaxis.
\newblock {\em Trans. Amer. Math. Soc.} 329,   819--824, 1992


\bibitem{22}
E.~F.~ Keller,  and L.~A.~ Segel.
\newblock
initiation of slime mold aggregation viewed as an instability.
\newblock {\em J. of Theor. Biol.}, 26, 399--415, 1970.



\bibitem{30}
H. Knutsdottir, E. Palsson, L. Edelstein-Keshet.
\newblock Mathematical model of macrophage-facilitated breast cancer cells invasion.
\newblock {\em J. of Theor. Biol.}, 357, 184--199, 2014.

\bibitem{24}
E.~H.~ Lieb.
\newblock Sharp cocntants in the Hardy-Littlewood-Sobolev and related inequalities.
\newblock {\em Ann. Math.}, 118(2), 349--374, 1983.

%
%
%
%
%
%
%
%
%
%
%




\bibitem{43}
Y.~Sugiyama.
\newblock Global existence in sub-critical cases and finite time blow-up in super-critical cases to degenerate Keller-Segel system.
\newblock {\em Differ. Integral Equ.}, 19, 841--876, 2006.


\bibitem{44}
Y.~Sugiyama.
\newblock Application of the best constant of the Sobolev inequality to degenerate Keller-Segel models.
\newblock {\em Adv. Differ.  Equ.}, 12, 121--144, 2007.


\bibitem{45}
Y.~Sugiyama,  H. Kunii.
\newblock Global existence and decay properties for a degenerate Keller-Segel model
with a power factor in drift term.
\newblock {\em J. Differential Equations}, 227, 333--364, 2006.


\bibitem{ulusoy}
S.~Ulusoy.
\newblock A Keller-Segel type system in Higher Dimensions.
\newblock {\em Annales de l'Institut Henri Poincar\'{e} (C) Analyse Non Lin\'{e}aire}, 34(4), 61--71, 2017.

%
%
%



\end{thebibliography}
\end{document}